\documentclass[a4paper,11pt,leqno]{smfart}

\usepackage{amssymb,amsmath,enumerate,verbatim,mathrsfs,graphics,graphicx,mathptm,float,mathtools}
\usepackage[T1]{fontenc}
\usepackage[english, francais]{babel}
\usepackage[all]{xy}
\usepackage[pdfstartview=FitH,
            colorlinks=true,
            linkcolor=blue,
            urlcolor=blue,
            citecolor=red,
            bookmarks,
            bookmarksopen=true,
            bookmarksnumbered=true]{hyperref}

\usepackage{cleveref}
\theoremstyle{plain}

\newtheorem{theoalph}{Theorem}

\newtheorem{thmalph}[theoalph]{Theorem}

\theoremstyle{definition}

\theoremstyle{remark}

\theoremstyle{plain}
\newtheorem{thmsec}{Theorem}[section]

\newtheorem{pro}[thmsec]{Proposition}
\newtheorem{lem}[thmsec]{Lemma}

\theoremstyle{definition}
\newtheorem{defin}[thmsec]{Definition}

\theoremstyle{remark}
\newtheorem{rem}[thmsec]{Remark}

\newtheorem{eg}[thmsec]{Example}

\def\og{\leavevmode\raise.3ex\hbox{$\scriptscriptstyle\langle\!\langle$~}}
\def\fg{\leavevmode\raise.3ex\hbox{~$\!\scriptscriptstyle\,\rangle\!\rangle$}}

\addtocounter{section}{0}             
\numberwithin{equation}{section}       

\newcommand{\Z}{\mathbb{Z}}

\newcommand{\C}{\mathbb{C}}

\newcommand{\pp}{\mathbb{P}^{2}_{\mathbb{C}}}

\newcommand\Sing{\mathrm{Sing}}

\newcommand\IF{\mathrm{I}_{\mathcal{F}}}
\newcommand\IinvF{\mathrm{I}_{\mathcal{F}}^{\mathrm{inv}}}
\newcommand\ItrF{\mathrm{I}_{\mathcal{F}}^{\hspace{0.2mm}\mathrm{tr}}}

\newcommand\IinvH{\mathrm{I}_{\mathcal{H}}^{\mathrm{inv}}}
\newcommand\ItrH{\mathrm{I}_{\mathcal{H}}^{\hspace{0.2mm}\mathrm{tr}}}
\newcommand\Cinv{\mathrm{C}_{\hspace{-0.3mm}\mathcal{H}}}
\newcommand\Dtr{\mathrm{D}_{\hspace{-0.3mm}\mathcal{H}}}

\newcommand\F{\mathcal{F}}

\newcommand\omegahesse{\omega_{\scalebox{0.64}{\ensuremath H}}^{4}}
\newcommand\Gunderline{{\mspace{2mu}\underline{\mspace{-2mu}\mathcal{G}\mspace{-2mu}}\mspace{2mu}}}

\newcommand\omegaoverline{{\mspace{2mu}\overline{\mspace{-1.4mu}\omega\mspace{-1.4mu}}\mspace{2mu}}}

\newcommand\Hesse{\mathcal{F}_{\hspace{-0.4mm}\raisebox{-0.2mm}{\tiny{$H$}}}^{4}}
\newcommand\elltext{\scalebox{1.1}{\ensuremath \ell}}
\newcommand\ellindice{\scalebox{0.9}{\ensuremath \ell}}

\begin{document}
\selectlanguage{english}

\title[Convex foliations of degree $4$ on the complex projective plane]{Convex foliations of degree $4$ on the complex projective plane}
\date{\today}

\author{Samir \textsc{Bedrouni}}

\address{Facult\'e de Math\'ematiques, USTHB, BP $32$, El-Alia, $16111$ Bab-Ezzouar, Alger, Alg\'erie}
\email{sbedrouni@usthb.dz}

\author{David \textsc{Mar\'{\i}n}}

\thanks{D. Mar\'{\i}n acknowledges financial support from the Spanish Ministry of Economy and Competitiveness, through grant MTM2015-66165-P and the "Mar\'{\i}a de Maeztu" Programme for Units of Excellence in R\&D (MDM-2014-0445).}

\address{Departament de Matem\`{a}tiques Universitat Aut\`{o}noma de Barcelona E-08193  Bellaterra (Barcelona) Spain}

\email{davidmp@mat.uab.es}


\maketitle{}

\begin{abstract}
We show that up to automorphism of $\pp$ there are $5$ homogeneous convex foliations of degree four on~$\pp.$ Using this result, we give a partial answer to a question posed in $2013$ by D.~\textsc{Mar\'{\i}n} and J. \textsc{Pereira} about the classification of reduced convex foliations on~$\pp.$

\end{abstract}

\section*{Introduction}
\bigskip

\noindent The set $\mathbf{F}(d)$ of foliations of degree $d$ on $\pp$ can be identified with a \textsc{Zariski} open subset of the projective space $\mathbb{P}_{\C}^{(d+2)^{2}-2}$. The group of automorphisms of $\pp$ acts on $\mathbf{F}(d)$. The orbit of an element $\F\in\mathbf{F}(d)$ under the action of
$\mathrm{Aut}(\pp)=\mathrm{PGL}_3(\mathbb{C})$ will be denoted by $\mathcal{O}(\F)$. Following \cite{MP13} we will say that a foliation in $\mathbf{F}(d)$ is \textsl{convex} if its leaves other than straight lines have no inflection points. The subset $\mathbf{FC}(d)$ of $\mathbf{F}(d)$ consisting of all convex foliations is \textsc{Zariski} closed in  $\mathbf{F}(d)$.

\noindent By \cite[Proposition~2, page~23]{Bru00} every foliation of degree~$0$ or~$1$ is convex, {\it i.e.} $\mathbf{FC}(0)=\mathbf{F}(0)$ and $\mathbf{FC}(1)=\mathbf{F}(0).$ For $d\geq2$, $\mathbf{FC}(d)$ is a proper closed subset of $\mathbf{F}(d)$ and it contains the \textsc{Fermat} foliation $\F_{0}^{d}$ of degree $d$, defined in the affine chart $(x,y)$ by the $1$-form (\emph{see}  \cite[page~179]{MP13}) $$\omegaoverline_{0}^{d}=(x^{d}-x)\mathrm{d}y-(y^{d}-y)\mathrm{d}x.$$

\noindent The closure inside $\mathbf{F}(d)$ of the orbit of $\F_{0}^{d}$ contains the foliations $\mathcal{H}_{0}^{d}$, resp. $\mathcal{H}_{1}^{d}$, resp. $\F_{1}^{d}$ (necessarily convex) defined by the $1$-forms (\emph{see} \cite[Example~6.5]{BM17} and \cite[page~75]{Bed17})
\begin{align*}
&\omega_{\hspace{0.2mm}0}^{\hspace{0.2mm}d}=(d-1)y^{d}\mathrm{d}x+x(x^{d-1}-dy^{d-1})\mathrm{d}y,&&
\text{resp.}\hspace{1.5mm}\omega_{\hspace{0.2mm}1}^{\hspace{0.2mm}d}=y^d\mathrm{d}x-x^d\mathrm{d}y,&&
\text{resp.}\hspace{1.5mm}\omegaoverline_{1}^{d}=y^{d}\mathrm{d}x+x^{d}(x\mathrm{d}y-y\mathrm{d}x).
\end{align*}
In other words, we have the following inclusions
\begin{equation}\label{equa:FC(d)}
\mathcal{O}(\mathcal{H}_{0}^{d})\cup\mathcal{O}(\mathcal{H}_{1}^{d})\cup\mathcal{O}(\F_{0}^{d})\cup\mathcal{O}(\F_{1}^{d})
\subset
\overline{\mathcal{O}(\F_{0}^{d})}
\subset
\mathbf{FC}(d).
\end{equation}
\noindent The foliations $\mathcal{H}_{0}^{d}$ and $\mathcal{H}_{1}^{d}$ are {\sl homogeneous}, {\it i.e.} they are invariant by homotheties; moreover, they are linearly conjugated for $d=2$, but not for $d\geq3$, \emph{see} \cite{BM17}. The dimension of the orbit of  $\F_{1}^{d}$ is~$6$ \cite{Bed17}, which is the least possible dimension in any degree  $d$ greater or equal to $2$ (\cite[Proposition~2.3]{CDGBM10}). Notice (\emph{see} \cite{Bed17}) that this bound is also attained by the non convex foliation  $\F_{2}^{d}$ defined by the $1$-form
 $$\omegaoverline_{2}^{d}=x^{d}\mathrm{d}x+y^{d}(x\mathrm{d}y-y\mathrm{d}x).$$

\noindent The classification of the elements of  $\mathbf{FC}(2)$ has been established by C.~\textsc{Favre} and J.~\textsc{Pereira} \cite[Proposition~7.4]{FP15}: up to automorphism of $\pp$, the foliations $\mathcal{H}_{0}^{2}$, $\F_{0}^{2}$ and $\F_{1}^{2}$ are the only convex foliations of degree $2$ on $\pp$.
This  classification implies that in degree $2$ the inclusions~(\ref{equa:FC(d)}) are equalities:
\begin{equation}\label{equa:FC(2)}
\mathbf{FC}(2)=\overline{\mathcal{O}(\F_{0}^{2})}=\mathcal{O}(\mathcal{H}_{0}^{2})\cup\mathcal{O}(\F_{0}^{2})\cup\mathcal{O}(\F_{1}^{2}).
\end{equation}
\noindent For dimensional reasons the orbits $\mathcal{O}(\F_{1}^{d})$ and $\mathcal{O}(\F_{2}^{d})$ are closed; combining equalities~(\ref{equa:FC(2)}) with \cite[Theorem~3]{CDGBM10}, we see, in particular, that the only closed orbits in  $\mathbf{F}(2)$ by the action of $\mathrm{Aut}(\pp)$ are those of $\F_{1}^{2}$ and $\F_{2}^{2}.$

\noindent Convex foliations of degree $3$ has been classified by the first author in his thesis \cite[Corollary~C]{Bed17}: every foliation $\F\in\mathbf{FC}(3)$ is linearly conjugated to one of the four foliations $\mathcal{H}_{0}^{3}$, $\mathcal{H}_{1}^{3}$, $\F_{0}^{3}$ or $\F_{1}^{3}$. This implies that the inclusions~(\ref{equa:FC(d)}) for $d=3$ are also equalities:
\begin{equation}\label{equa:FC(3)}
\mathbf{FC}(3)=\overline{\mathcal{O}(\F_{0}^{3})}=\mathcal{O}(\mathcal{H}_{0}^{3})\cup\mathcal{O}(\mathcal{H}_{1}^{3})
\cup\mathcal{O}(\F_{0}^{3})\cup\mathcal{O}(\F_{1}^{3}).
\end{equation}

\noindent For $d\geq4$, the classification of the elements of $\mathbf{FC}(d)$ modulo $\mathrm{Aut}(\pp)$ remains open and the topological structure of $\mathbf{FC}(d)$ is not yet well understood. In the sequel we will focus on the case $d=4$. Notice  (\emph{see}~\cite[page~181]{MP13}) that the set $\mathbf{FC}(4)$ contains the foliation $\Hesse$, called  \textsc{Hesse} pencil of degree~$4$, defined by $$\omegahesse=(2x^{3}-y^{3}-1)y\mathrm{d}x+(2y^{3}-x^{3}-1)x\mathrm{d}y\hspace{1mm};$$ furthermore $\mathcal{O}(\Hesse)\neq\mathcal{O}(\F_{0}^{4})$ and $\dim\mathcal{O}(\Hesse)=\dim\mathcal{O}(\F_{0}^{4})=8$. So that the inclusion $\overline{\mathcal{O}(\F_{0}^{4})}\subset\mathbf{FC}(4)$ is strict, in contrast to the previous cases of degrees $2$ and $3.$

\noindent In this paper we propose to classify, up to automorphism, the foliations of $\mathbf{FC}(4)$ which are homogeneous. More precisely, we establish the following theorem.
\begin{thmalph}\label{thmalph:class-homogene-convexe-4}
{\sl Up to automorphism of $\pp$ there are five homogeneous convex foliations of degree four $\mathcal{H}_1,\ldots,\mathcal{H}_{5}$ on the complex projective plane. They are respectively described in affine chart by the following $1$-forms

\begin{itemize}
\item [\texttt{1.}]\hspace{1mm} $\omega_1=y^4\mathrm{d}x-x^4\mathrm{d}y$;

\item [\texttt{2.}]\hspace{1mm} $\omega_2=y^3(2x-y)\mathrm{d}x+x^3(x-2y)\mathrm{d}y$;

\item [\texttt{3.}]\hspace{1mm} $\omega_3=y^2(6x^2+4xy+y^2)\mathrm{d}x-x^3(x+4y)\mathrm{d}y$;

\item [\texttt{4.}]\hspace{1mm} $\omega_4=y^3(4x+y)\mathrm{d}x+x^3(x+4y)\mathrm{d}y$;

\item [\texttt{5.}]\hspace{1mm} $\omega_5=y^2(6x^2+4xy+y^2)\mathrm{d}x+3x^4\mathrm{d}y$.
\end{itemize}
}
\end{thmalph}

\noindent By \cite{Per01} we know that every foliation of degree $d\geq 1$ on $\pp$ can not have more than $3d$ (distinct) invariant lines. If this bound is reached for  $\F\in\mathbf{F}(d)$, then $\F$ necessarily belongs to~$\mathbf{FC}(d)$; in this case we say that $\F$ is  \textsl{reduced convex}. To our knowledge the only reduced convex foliations known in the literature are those presented in
 \cite[Table~1.1]{MP13}: the \textsc{Fermat} foliation $\F_{0}^{d}$ in any degree, the  \textsc{Hesse} pencil $\Hesse$ and the foliations given by the $1$-forms
\[(y^2-1)(y^2- (\sqrt{5}-2)^2)(y+\sqrt{5}x) \mathrm{d}x-(x^2-1)(x^2- (\sqrt{5}-2)^2)(x+\sqrt{5}y) \mathrm{d}y \, ,\]
\[(y^3-1)(y^3+7x^3+1) y \mathrm{d}x-(x^3-1)(x^3+7 y^3+1) x \mathrm{d}y\,,\]
which have degrees $5$ and $7$ respectively. D.~\textsc{Mar\'{\i}n} and J. \textsc{Pereira} \cite[Problem~9.1]{MP13} asked the following question: are there other reduced convex foliations? The answer in degree $2$, resp. $3$, to this question is negative, by \cite[Proposition~7.4]{FP15}, resp. \cite[Corollary~6.9]{BM17}. Theorem~\ref{thmalph:class-homogene-convexe-4} allows us to show that the answer to \cite[Problem~9.1]{MP13} in degree $4$ is also negative.
\begin{thmalph}\label{thmalph:convexe-reduit-4}
{\sl Up to automorphism of $\pp$, the \textsc{Fermat} foliation $\F_{0}^{4}$ and the \textsc{Hesse} pencil $\Hesse$ are the only reduced convex foliations of degree four on $\pp.$}
\end{thmalph}

\section{Preliminaries}
\bigskip

\subsection{Singularities and inflection divisor of a foliation on the projective plane}

A degree $d$ holomorphic foliation $\F$ on $\pp$ is defined in homogeneous coordinates $[x:y:z]$ by a $1$-form $$\omega=a(x,y,z)\mathrm{d}x+b(x,y,z)\mathrm{d}y+c(x,y,z)\mathrm{d}z,$$ where $a,$ $b$ and $c$ are homogeneous polynomials of degree $d+1$ without common factor and satisfying the \textsc{Euler} condition $i_{\mathrm{R}}\omega=0$, where $\mathrm{R}=x\frac{\partial{}}{\partial{x}}+y\frac{\partial{}}{\partial{y}}+z\frac{\partial{}}{\partial{z}}$ denotes the radial vector field and  $i_{\mathrm{R}}$ is the interior product by  $\mathrm{R}$. The {\sl singular locus} $\mathrm{Sing}\mathcal{F}$ of $\mathcal{F}$ is the projectivization of the singular locus of~$\omega$ $$\mathrm{Sing}\,\omega=\{(x,y,z)\in\mathbb{C}^3\,\vert \, a(x,y,z)=b(x,y,z)=c(x,y,z)=0\}.$$

\noindent Let us recall some local notions attached to the pair $(\mathcal{F},s)$, where $s\in\Sing\mathcal{F}$. The germ of $\F$ at $s$ is defined, up to multiplication by a unity in the local ring $\mathcal{O}_s$ at $s$, by a vector field $\mathrm{X}=A(\mathrm{u},\mathrm{v})\frac{\partial{}}{\partial{\mathrm{u}}}+B(\mathrm{u},\mathrm{v})\frac{\partial{}}{\partial{\mathrm{v}}}$. The~{\sl vanishing order} $\nu(\mathcal{F},s)$ of $\mathcal{F}$ at $s$ is given by $$\nu(\mathcal{F},s)=\min\{\nu(A,s),\nu(B,s)\},$$ where $\nu(g,s)$ denotes the vanishing order of the function $g$ at $s$. The {\sl tangency order} of $\mathcal{F}$ with a generic line passing through $s$ is the integer
$$\hspace{0.8cm}\tau(\mathcal{F},s)=\min\{k\geq\nu(\mathcal{F},s)\hspace{1mm}\colon\det(J^{k}_{s}\,\mathrm{X},\mathrm{R}_{s})\neq0\},$$ where $J^{k}_{s}\,\mathrm{X}$ denotes the $k$-jet of $\mathrm{X}$ at $s$ and $\mathrm{R}_{s}$ is the radial vector field centered at $s.$ The \textsl{\textsc{Milnor} number} of $\F$ at $s$ is the integer $$\mu(\mathcal{F},s)=\dim_\mathbb{C}\mathcal{O}_s/\langle A,B\rangle,$$ where $\langle A,B\rangle$ denotes the ideal of $\mathcal{O}_s$ generated by $A$ and $B.$
\smallskip

\noindent The singularity $s$ is called {\sl radial of order} $n-1$ if $\nu(\mathcal{F},s)=1$ and $\tau(\mathcal{F},s)=n.$
\smallskip

\noindent The singularity $s$ is called {\sl non-degenerate} if  $\mu(\F,s)=1$, or equivalently if the linear part  $J^{1}_{s}\mathrm{X}$ of $\mathrm{X}$ possesses two non-zero eigenvalues $\lambda,\mu.$ In this case, the quantity  $\mathrm{BB}(\F,s)=\frac{\lambda}{\mu}+\frac{\mu}{\lambda}+2$ is called the {\sl  \textsc{Baum-Bott} invariant} of $\F$ at $s$ (\emph{see} \cite{BB72}). By \cite{CS82} there is at least a germ of curve
$\mathcal{C}$ at $s$ which is invariant by $\F$. Up to local diffeomorphism we can assume that $s=(0,0)$\, $\mathrm{T}_{s}\mathcal{C}=\{\mathrm{v}=\,0\,\}$ and $J^{1}_{s}\mathrm{X}=\lambda \mathrm{u}\frac{\partial}{\partial\mathrm{u}}+(\varepsilon \mathrm{u}+\mu\hspace{0.1mm}\mathrm{v})\frac{\partial}{\partial \mathrm{v}}$, where we can take $\varepsilon=0$ if $\lambda\neq\mu$. The quantity $\mathrm{CS}(\F,\mathcal{C},s)=\frac{\lambda}{\mu}$ is called the {\sl \textsc{Camacho-Sad} index} of $\F$ at $s$ along~$\mathcal{C}.$
\medskip

\noindent Let us also recall the notion of inflection divisor of $\F$. Let $\mathrm{Z}=E\frac{\partial}{\partial x}+F\frac{\partial}{\partial y}+G\frac{\partial}{\partial z}$ be a homogeneous vector field of degree $d$ on  $\mathbb{C}^3$ non collinear to the radial vector field describing $\mathcal{F},$ {\it i.e.} such that $\omega=i_{\mathrm{R}}i_{\mathrm{Z}}\mathrm{d}x\wedge\mathrm{d}y\wedge\mathrm{d}z.$ The {\sl inflection divisor} of $\mathcal{F}$, denoted by $\IF$, is the divisor of $\pp$ defined by the homogeneous equation
\begin{equation*}
\left| \begin{array}{ccc}
x &  E &  \mathrm{Z}(E) \\
y &  F &  \mathrm{Z}(F)  \\
z &  G &  \mathrm{Z}(G)
\end{array} \right|=0.
\end{equation*}
This divisor has been studied in \cite{Per01} in a more general context. In particular, the following properties has been proved.
\begin{enumerate}
\item [\texttt{1.}] On $\mathbb{P}^{2}_{\mathbb{C}}\smallsetminus\mathrm{Sing}\mathcal{F},$ $\IF$ coincides with the curve described by the inflection
points of the leaves of $\mathcal{F}$;
\item [\texttt{2.}] If $\mathcal{C}$ is an irreducible algebraic curve invariant by $\mathcal{F}$ then $\mathcal{C}\subset\IF$ if and only if $\mathcal{C}$ is an invariant line;
\item [\texttt{3.}] $\IF$ can be decomposed into $\IF=\IinvF+\ItrF,$ where the support of $\IinvF$ consists in the set of invariant lines of $\F$ and the support of $\ItrF$ is the closure of the isolated inflection points along the leaves of  $\mathcal{F}$;
\item [\texttt{4.}] The degree of the divisor $\IF$ is $3d.$
\end{enumerate}

\noindent The foliation $\mathcal{F}$ will be called {\sl convex} if its inflection divisor $\IF$ is totally invariant by $\mathcal{F}$, {\it i.e.} if $\IF$
is a product of invariant lines.

\subsection{Geometry of homogeneous foliations}

A foliation of degree $d$ on $\pp$ is said to be {\sl homogeneous} if there is an affine chart $(x,y)$ of $\mathbb{P}^{2}_{\mathbb{C}}$ in which it is invariant under the action of the group of homotheties $(x,y)\longmapsto \lambda(x,y),\hspace{1mm} \lambda\in \mathbb{C}^{*}.$ Such a foliation $\mathcal{H}$ is then defined by a $1$-form $$\omega=A(x,y)\mathrm{d}x+B(x,y)\mathrm{d}y,$$ where $A$ and $B$ are homogeneous polynomials of degree $d$ without common factor. This $1$-form writes in homogeneous coordinates as
\begin{align*}
&\hspace{2mm} z\hspace{0.3mm}A(x,y)\mathrm{d}x+z\hspace{0.2mm}B(x,y)\mathrm{d}y-\left(x\hspace{0.2mm}A(x,y)+yB(x,y)\right)\mathrm{d}z.
\end{align*}
Thus the foliation $\mathcal{H}$ has at most $d+2$ singularities whose origin $O$ of the affine chart $z=1$ is the only singular point of $\mathcal{H}$ which is not situated on the line at infinity $L_{\infty}=\{z=0\}$; moreover $\nu(\mathcal{H},O)=d.$

\noindent In the sequel we will assume that $d$ is greater than or equal to $2.$ In this case the point $O$ is the only singularity of $\mathcal{H}$ having vanishing order $d.$

\noindent We know from~\cite{BM17} that the inflection divisor of $\mathcal{H}$ is given by $z\Cinv\Dtr=0,$ where $\Cinv=xA+yB\in\mathbb{C}[x,y]_{d+1}$ denotes the {\sl tangent cone} of $\mathcal{H}$ at the origin $O$ and $\Dtr=\frac{\partial A}{\partial x}\frac{\partial B}{\partial y}-\frac{\partial A}{\partial y}\frac{\partial B}{\partial x}\in\mathbb{C}[x,y]_{2d-2}.$ From this we deduce~that:
\begin{enumerate}
\item [(i)]  the support of the divisor $\IinvH$ consists of the lines of the tangent cone $\Cinv=0$ and the line at infinity~$L_{\infty}$~;
\item [(ii)] the divisor $\ItrH$ decomposes as $\ItrH=\prod_{i=1}^{n}T_{i}^{\rho_{i}-1}$ for some number $n\leq \deg\Dtr=2d-2$ of lines $T_{i}$ passing through $O,$
    $\rho_{i}-1$ being the inflection order of the line $T_{i}.$
\end{enumerate}

\begin{pro}[\cite{BM17}, \rm{Proposition~2.2}]\label{pro:SingH}
{\sl With the previous notations, for any point $s\in\mathrm{Sing}\mathcal{H}\cap L_{\infty}$, we~have

\noindent\textbf{\textit{1.}} $\nu(\mathcal{H},s)=1;$

\noindent\textbf{\textit{2.}} the line joining the origin $O$ to the point $s$ is invariant by $\mathcal{H}$ and it appears with multiplicity $\tau(\mathcal{H},s)-1$ in the divider $\Dtr=0,$ {\it i.e.} $$\Dtr=\ItrH\prod_{s\in\mathrm{Sing}\mathcal{H}\cap L_{\infty}}L_{s}^{\tau(\mathcal{H},s)-1}.$$
}
\end{pro}

\begin{defin}[\cite{BM17}]\label{def:type-homog}
Let $\mathcal{H}$ be a homogeneous foliation of degree $d$ on $\pp$ having a certain number $m\leq d+1$ of radial singularities $s_{i}$ of order $\tau_{i}-1,$ $2\leq\tau_{i}\leq d$ for $i=1,2,\ldots,m.$ The support of the divisor $\ItrH$ consists of a certain number $n\leq2d-2$ of transverse inflection lines $T_{j}$ of order $\rho_{j}-1,$ $2\leq\rho_{j}\leq d$ for $j=1,2,\ldots,n.$ We define the {\sl type of the foliation} $\mathcal{H}$ by
$$
\mathcal{T}_\mathcal{H}=\sum\limits_{i=1}^{m}\mathrm{R}_{\tau_{i}-1}+\sum\limits_{j=1}^{n}\mathrm{T}_{\rho_{j}-1}=
\sum\limits_{k=1}^{d-1}(r_{k}\cdot\mathrm{R}_k+t_{k}\cdot\mathrm{T}_k)
\in\Z\left[\mathrm{R}_1,\mathrm{R}_2,\ldots,\mathrm{R}_{d-1},\mathrm{T}_1,\mathrm{T}_2,\ldots,\mathrm{T}_{d-1}\right].
$$
\end{defin}
\smallskip

\begin{eg}
Let us consider the homogeneous foliation $\mathcal{H}$ of degree $5$ on $\pp$ defined by $$\omega=y^5\mathrm{d}x+2x^3(3x^2-5y^2)\mathrm{d}y.$$ A straightforward computation leads to
\begin{align*}
&& \mathrm{C}_{\mathcal{H}}=xy\left(6x^4-10x^2y^2+y^4\right) \qquad\text{and}\qquad \Dtr=150\hspace{0.15mm}x^2y^4(x-y)(x+y).
\end{align*}
We see that the set of radial singularities of $\mathcal{H}$ consists of the two points $s_1=[0:1:0]$ and $s_2=[1:0:0]$; their orders of radiality are equal to $2$ and $4$ respectively. Moreover the support of the divisor $\ItrH$ is the union of the two lines $x-y=0$ and $x+y=0$; they are transverse inflection lines of order~$1.$ Therefore the foliation $\mathcal{H}$ is of type $\mathcal{T}_\mathcal{H}=1\cdot\mathrm{R}_{2}+1\cdot\mathrm{R}_{4}+2\cdot\mathrm{T}_{1}.$
\end{eg}

\noindent Following \cite{BM17} to every homogeneous foliation $\mathcal{H}$ of degree $d$ on $\pp$ is associated a rational map from the \textsc{Riemann} sphere $\mathbb{P}^{1}_{\mathbb{C}}$ to itself of degree $d$ denoted by $\Gunderline_{\mathcal{H}}$ and defined as follows: if $\mathcal{H}$ is described by $\omega=A(x,y)\mathrm{d}x+B(x,y)\mathrm{d}y,$ with $A$ and $B$ being homogeneous polynomials of degree $d$ without common factor, the image of the point $[x:y]\in\mathbb{P}^{1}_{\mathbb{C}}$ by $\Gunderline_{\mathcal{H}}$ is the point $[-A(x,y):B(x,y)]\in\mathbb{P}^{1}_{\mathbb{C}}.$ It is clear that this definition does not depend on the choice of the homogeneous $1$-form $\omega$ describing the foliation $\mathcal{H}.$ Notice that the map $\Gunderline_{\mathcal{H}}$ has the following properties (\emph{see} \cite{BM17}):
\begin{enumerate}
\item [\texttt{1.}] the fixed points of $\Gunderline_{\mathcal{H}}$ correspond to the tangent cone of $\mathcal{H}$ at the origin $O$ (\textit{i.e.} $[a:b]\in\mathbb{P}^{1}_{\mathbb{C}}$ is fixed by $\Gunderline_{\mathcal{H}}$ if and only if the line $by-a\hspace{0.2mm}x=0$ is invariant by $\mathcal{H}$\hspace{0.25mm});

\item [\texttt{2.}] the point $[a:b]\in\mathbb{P}^{1}_{\mathbb{C}}$ is a fixed critical point of $\Gunderline_{\mathcal{H}}$ if and only if the point $[b:a:0]\in L_{\infty}$ is a radial singularity of $\mathcal{H}$. The multiplicity of the critical point $[a:b]$ of $\Gunderline_{\mathcal{H}}$ is exactly equal to the the radiality order of the singularity at infinity;

\item [\texttt{3.}] the point $[a:b]\in\mathbb{P}^{1}_{\mathbb{C}}$ is a non-fixed critical point of  $\Gunderline_{\mathcal{H}}$ if and only if the line $by-a\hspace{0.2mm}x=0$ is a transverse inflection line of  $\mathcal{H}.$ The multiplicity of the critical point $[a:b]$ of $\Gunderline_{\mathcal{H}}$ is precisely equal to the inflection order of this line.
\end{enumerate}
\smallskip

\noindent It follows, in particular, that a homogeneous foliation  $\mathcal{H}$ on $\pp$ is convex if and only if its associated map~$\Gunderline_{\mathcal{H}}$ has only fixed critical points; more precisely, a homogeneous foliation $\mathcal{H}$ of degree $d$ on $\pp$ is convex of type  $\mathcal{T}_\mathcal{H}=\sum_{k=1}^{d-1}r_{k}\cdot\mathrm{R}_k$ if and only if the map $\Gunderline_{\mathcal{H}}$ possesses  $r_{1}$, resp.~$r_{2},\ldots,$ resp.~$r_{d-1}$ fixed critical points of multiplicity $1$, resp. $2\ldots,$ resp. $d-1,$ with $\sum_{k=1}^{d-1}kr_{k}=2d-2.$

\section{Proof of Theorem~\ref{thmalph:class-homogene-convexe-4}}

\bigskip

\noindent Before proving  Theorem~\ref{thmalph:class-homogene-convexe-4}, let us recall the next result which follows from Propositions~$4.1$~and~$4.2$ of~\cite{BM17}:
\begin{pro}[\cite{BM17}]\label{pro:omega1-omega3}
{\sl Let  $\mathcal{H}$ be a convex homogeneous foliation of degree  $d\geq3$ on $\pp$. Let $\nu$ be an integer between $1$ and $d-2.$ Then, $\mathcal{H}$ is of type
\[
\hspace{-1.7cm}\mathcal{T}_{\mathcal{H}}=2\cdot\mathrm{R}_{d-1},
\hspace{0.6cm}{\qquad\qquad\text{resp}.\hspace{1.5mm}
\mathcal{T}_{\mathcal{H}}=1\cdot\mathrm{R}_{\nu}+1\cdot\mathrm{R}_{d-\nu-1}+1\cdot\mathrm{R}_{d-1}},
\]
if and only if it is linearly conjugated to the foliation
 $\mathcal{H}_{1}^{d}$, resp. $\mathcal{H}_{3}^{d,\nu}$ given by
\[
\omega_1^{\hspace{0.2mm}d}=y^d\mathrm{d}x-x^d\mathrm{d}y,
\qquad\qquad\text{resp}.\hspace{1.5mm}
\omega_{3}^{\hspace{0.2mm}d,\nu}=\sum\limits_{i=\nu+1}^{d}\binom{{d}}{{i}}x^{d-i}y^i\mathrm{d}x-\sum\limits_{i=0}^{\nu}\binom{{d}}{{i}}x^{d-i}y^i\mathrm{d}y.
\]
}
\end{pro}

\begin{proof}[\sl Proof of Theorem~\ref{thmalph:class-homogene-convexe-4}]
Let  $\mathcal{H}$  be a convex homogeneous foliation of degree $4$
on $\pp$, defined in the affine chart  $(x,y)$, by the $1$-form $$\hspace{1cm}\omega=A(x,y)\mathrm{d}x+B(x,y)\mathrm{d}y,\quad A,B\in\mathbb{C}[x,y]_4,\hspace{2mm}\gcd(A,B)=1.$$
By \cite[Remark~2.5]{Bed17} the foliation  $\mathcal{H}$ can not have  $4+1=5$ distinct radial singularities,  in other words it can not be of type  $4\cdot\mathrm{R}_{1}+1\cdot\mathrm{R}_{2}.$
We are then in one of the following situations:
\begin{align*}
&\mathcal{T}_{\mathcal{H}}=2\cdot\mathrm{R}_{3}\hspace{0.5mm};&&
\mathcal{T}_{\mathcal{H}}=1\cdot\mathrm{R}_{1}+1\cdot\mathrm{R}_{2}+1\cdot\mathrm{R}_{3}\hspace{0.5mm};&&
\mathcal{T}_{\mathcal{H}}=3\cdot\mathrm{R}_{2}\hspace{0.5mm};\\
&\mathcal{T}_{\mathcal{H}}=2\cdot\mathrm{R}_{1}+2\cdot\mathrm{R}_{2}\hspace{0.5mm};&&
\mathcal{T}_{\mathcal{H}}=3\cdot\mathrm{R}_{1}+1\cdot\mathrm{R}_{3}.
\end{align*}
\begin{itemize}
\item [$\bullet$] If $\mathcal{T}_{\mathcal{H}}=2\cdot\mathrm{R}_{3}$, resp. $\mathcal{T}_{\mathcal{H}}=1\cdot\mathrm{R}_{1}+1\cdot\mathrm{R}_{2}+1\cdot\mathrm{R}_{3}$, then by \cite[Propositions~4.1,~4.2]{BM17}, the $1$-form $\omega$ is linearly conjugated to
    \[
      \hspace{-9.1cm}\omega_1^{\hspace{0.2mm}4}=y^4\mathrm{d}x-x^4\mathrm{d}y=\omega_1,
    \]
    \[
      \text{resp}.\hspace{1.5mm}
      \omega_{3}^{\hspace{0.2mm}4,1}=\,\sum\limits_{i=2}^{4}\binom{{4}}{{i}}x^{4-i}y^i\mathrm{d}x-
      \sum\limits_{i=0}^{1}\binom{{4}}{{i}}x^{4-i}y^i\mathrm{d}y\,=\,y^2(6x^2+4xy+y^2)\mathrm{d}x-x^3(x+4y)\mathrm{d}y\,=\,\omega_3.
    \]

\item [$\bullet$] Assume that $\mathcal{T}_{\mathcal{H}}=3\cdot\mathrm{R}_{2}$.
This means that the rational map $\Gunderline_{\mathcal{H}}\hspace{1mm}\colon\mathbb{P}^{1}_{\mathbb{C}}\rightarrow \mathbb{P}^{1}_{\mathbb{C}}$, $\Gunderline_{\mathcal{H}}(z)~=~-\dfrac{A(1,z)}{B(1,z)},$~admits three different fixed critical  points of multiplicity $2$.
By \cite[page~79]{CGNPP15}, $\Gunderline_{\mathcal{H}}$ is conjugated by a \textsc{Möbius} transformation to $z\mapsto-\dfrac{z^3(2-z)}{1-2z}$.
As a consequence, $\omega$ is linearly conjugated to $$\omega_2=y^3(2x-y)\mathrm{d}x+x^3(x-2y)\mathrm{d}y.$$
\vspace{2mm}
\item [$\bullet$] Assume that $\mathcal{T}_{\mathcal{H}}=2\cdot\mathrm{R}_{1}+2\cdot\mathrm{R}_{2}$. Then the rational map $\Gunderline_{\mathcal{H}}$ possesses four fixed critical points, two of them having multiplicity $1$ and the other two having multiplicity $2$. This implies, by \cite[page~79]{CGNPP15}, that up to conjugation by a \textsc{Möbius} transformation, $\Gunderline_{\mathcal{H}}$ writes as
    \[
    \hspace{1cm}z\mapsto-\frac{z^3(2z+3cz-4c-3)}{z+c},
    \]
     where $c=-3/8\pm\sqrt{5}/8$. Thus, up to linear conjugation
    \[
    \hspace{1cm}\omega=y^3(2y+3cy-4cx-3x)\mathrm{d}x+x^3(y+cx)\mathrm{d}y,\qquad c=-\frac{3}{8}\pm\frac{\sqrt{5}}{8}.
    \]
    In both cases  ($c=-3/8+\sqrt{5}/8$ or $c=-3/8-\sqrt{5}/8$), the $1$-form $\omega$ is linearly  conjugated to
     \[
     \omega_4=y^3(4x+y)\mathrm{d}x+x^3(x+4y)\mathrm{d}y\hspace{0.5mm}.
     \]
Indeed,
     \[
       \hspace{0.5cm}\omega_{4}=\frac{\raisebox{-0.5mm}{$3c+2$}}{2}\varphi^*\omega,\quad \text{where}\hspace{1.5mm} \varphi=(2x,8cy).
     \]
\item [$\bullet$] Finally, consider the last situation: $\mathcal{T}_{\mathcal{H}}=3\cdot\mathrm{R}_{1}+1\cdot\mathrm{R}_{3}.$  Up to linear conjugation we can assume that   $\Dtr=cx^3y(y-x)(y-\alpha x)$ and $\Cinv(0,1)=\Cinv(1,0)=\Cinv(1,1)=\Cinv(1,\alpha)=0$, for some  $c,\alpha\in \mathbb{C}^*,\alpha\neq1$. The  points $\infty=[1:0],\,[0:1],\,[1:1],\,[1:\alpha]\in\mathbb{P}^{1}_{\mathbb{C}}$ are then fixed and critical for $\Gunderline_{\mathcal{H}}$, having respective multiplicities  $3,1,1,1$. By \cite[Lemma~3.9]{BM17}, there exist constants $a_0,a_2,b\in\C^{*},a_1\in\C$ such that
    \begin{align*}
    &\hspace{3mm}B(x,y)=bx^4,&& A(x,y)=(a_0\hspace{0.2mm}x^2+a_1xy+a_2y^2)y^2,&&
    (z-1)^2\hspace{1mm} \text{divides}\hspace{1mm} P(z),&&
    (z-\alpha)^2\hspace{1mm} \text{divides}\hspace{1mm} Q(z),
    \end{align*}
    where $P(z):=A(1,z)+B(1,z)$ and $Q(z):=A(1,z)+\alpha B(1,z).$ It follows that
    \begin{align*}\left\{
    \begin{array}[c]{l}
    \vspace{2.8mm}
    P(1)=0\\
    \vspace{2.8mm}
    P^{\prime}(1)=0\\
    \vspace{2.8mm}
    Q(\alpha)=0\\
    Q^{\prime}(\alpha)=0
    \end{array}
    \right.
    \Leftrightarrow\hspace{2mm}
    \left\{\begin{array}{llll}
    \vspace{2.8mm}
    a_0+a_1+a_2+b=0\\
    \vspace{2.8mm}
    2a_0+3a_1+4a_2=0\\
    \vspace{2.8mm}
    a_2\alpha^3+a_1\alpha^2+a_0\alpha+b=0\\
    4a_2\alpha^2+3a_1\alpha+2a_0=0\\
    \end{array}\right.
    \Leftrightarrow\hspace{2mm}
    \left\{\begin{array}{llll}
    a_0=2a_2\alpha\\
    a_1=-\dfrac{4a_2(\alpha+1)}{3}\\
    b=-\dfrac{a_2(2\alpha-1)}{3}\\
    \alpha^2-\alpha+1=0\\
    \end{array}\right.
    \end{align*}
    By replacing $\omega$ by $\dfrac{\raisebox{-0.8mm}{$3$}}{\raisebox{1mm}{$a_2$}}\omega,$ we can assume that
    \[
    \hspace{1cm}\omega=y^2(6\alpha x^2-4(\alpha+1)xy+3y^2)\mathrm{d}x-(2\alpha-1)x^4\mathrm{d}y,\qquad \alpha^2-\alpha+1=0.
    \]
    The $1$-form $\omega$ is linearly conjugated to
    \[
     \omega_5=y^2(6x^2+4xy+y^2)\mathrm{d}x+3x^4\mathrm{d}y\hspace{0.5mm}.
    \]
    Indeed, the fact that  $\alpha$ satisfies $\alpha^2-\alpha+1=0$ implies that
    \[
    \hspace{1.9cm}\omega_{5}=\frac{\raisebox{-0.5mm}{$1-\alpha$}}{(\alpha-2)^3}\varphi^*\omega,
    \quad\text{where}\hspace{1.5mm}
    \varphi=\Big((\alpha-2)x,y\Big).
    \]
\end{itemize}

\noindent The foliations $\mathcal{H}_{\hspace{0.2mm}i},i=1,\ldots,5,$ are not linearly conjugated because, by construction, $\mathcal{T}_{\mathcal{H}_{j}}~\neq~\mathcal{T}_{\mathcal{H}_{\hspace{0.2mm}i}}$ for each $j\neq i.$ This ends the proof of the theorem.
\end{proof}
\smallskip

\noindent A remarkable feature of the classification obtained is that all the singularities of the foliations $\mathcal{H}_{\hspace{0.2mm}i},i=1,\ldots,5,$ on the line at infinity are non-degenerated. In the following section we will need the values of the \textsc{Camacho-Sad} indices $\mathrm{CS}(\mathcal{H}_{\hspace{0.2mm}i},L_{\infty},s)$, $s\in\Sing\mathcal{H}_{\hspace{0.2mm}i}\cap L_{\infty}$. For this reason, we have computed, for each $i=1,\ldots,5$, the following polynomial (called \textsl{\textsc{Camacho-Sad} polynomial of the homogeneous foliation}  $\mathcal{H}_{\hspace{0.2mm}i}$)
\begin{align*}
\mathrm{CS}_{\mathcal{H}_{\hspace{0.2mm}i}}(\lambda)=\prod\limits_{s\in\Sing\mathcal{H}_{\hspace{0.2mm}i}\cap L_{\infty}}(\lambda-\mathrm{CS}(\mathcal{H}_{\hspace{0.2mm}i},L_{\infty},s)).
\end{align*}
\noindent The following table summarizes the types and  the \textsc{Camacho-Sad} polynomials of the foliations $\mathcal{H}_{\hspace{0.2mm}i}$, $i=1,\ldots,5$.

\begingroup
\renewcommand*{\arraystretch}{1.5}
\begin{table}[h]
\begin{center}
\begin{tabular}{|c|c|c|}\hline
$i$  & $\mathcal{T}_{\mathcal{H}_{\hspace{0.2mm}i}}$              &  $\mathrm{CS}_{\mathcal{H}_{\hspace{0.2mm}i}}(\lambda)$ \\\hline
$1$  & $2\cdot\mathrm{R}_3$                                       &  $(\lambda-1)^{2}(\lambda+\frac{1}{3})^{3}$              \\\hline
$2$  & $3\cdot\mathrm{R}_2$                                       &  $(\lambda-1)^{3}(\lambda+1)^{2}$                         \\\hline
$3$  & $1\cdot\mathrm{R}_1+1\cdot\mathrm{R}_2+1\cdot\mathrm{R}_3$ &  $(\lambda-1)^{3}(\lambda+\frac{13+2\sqrt{13}}{13})(\lambda+\frac{13-2\sqrt{13}}{13})$ \\\hline
$4$  & $2\cdot\mathrm{R}_1+2\cdot\mathrm{R}_2$                    &  $(\lambda-1)^{4}(\lambda+3)$ \\\hline
$5$  & $3\cdot\mathrm{R}_1+1\cdot\mathrm{R}_3$                    &  $(\lambda-1)^{4}(\lambda+3)$ \\\hline
\end{tabular}
\end{center}
\bigskip
\caption{Types and \textsc{Camacho-Sad} polynomials of the homogeneous foliations given by Theorem~\ref{thmalph:class-homogene-convexe-4}.}\label{tab:CS(lambda)}
\end{table}
\endgroup

\section{Proof of Theorem~\ref{thmalph:convexe-reduit-4}}

\bigskip

\noindent The proof of Theorem~\ref{thmalph:convexe-reduit-4}  is based on the classification of convex homogeneous foliations of degree four on $\pp$ given by
Theorem~\ref{thmalph:class-homogene-convexe-4} and on the three following results which hold in arbitrary degree.
\smallskip

\noindent First, notice that if  $\F$ is a foliation of degree  $d\geq1$ on $\pp$ and if  $s$ is a singular point of $\F$ then  $$\sigma(\F,s)\leq\tau(\F,s)+1\leq d+1,$$ where $\sigma(\F,s)$ denotes the number of (distinct) invariant lines of $\F$ passing through $s$.

\noindent The following lemma shows that the left-hand inequality above is an equality in the case where $\F$ is a reduced convex foliation.

\begin{lem}\label{lem:sigma(F,s)}
{\sl
Let $\F$ be a reduced convex foliation of degree  $d\geq1$ on $\pp.$  Then, through each singular point $s$ of $\F$ pass exactly $\tau(\F,s)+1$ invariant lines of~$\F$, {\it i.e.} $\sigma(\F,s)=\tau(\F,s)+1.$
}
\end{lem}

\begin{proof}
Let  $s$  be a singular point of $\F$. Since the inflection divisor  $\IF$ of $\F$ is totally invariant by $\F$ and it is reduced, we deduce that $\mu(\F,s)=1$ (\cite[Lemma~6.8]{BM17}) and the number  $\sigma(\F,s)$ coincides with the vanishing order of $\IF$ at $s.$ On the other hand, an elementary computation, using the equality $\mu(\F,s)=1,$ shows that the vanishing order of $\IF$ at $s$ is equal to~$\tau(\F,s)+1.$ Hence the lemma holds.
\end{proof}

\noindent The following result allows us to reduce the study of the convexity to the homogeneous framework:
\begin{pro}\label{pro:F-degenere-H}
{\sl Let $\F$ be a reduced convex foliation of degree $d\geq1$ on $\pp$ and let $\elltext$ be one of its  $3d$ invariant lines. There is a convex homogeneous foliation $\mathcal{H}$ of degree $d$ on $\pp$ satisfying the following properties:
\vspace{1mm}
\begin{itemize}
\item [(i)] $\mathcal{H}\in\overline{\mathcal{O}(\F)}$;
\vspace{0.5mm}
\item [(ii)] $\elltext$ is invariant by $\mathcal{H}$;
\vspace{0.5mm}
\item [(iii)] $\Sing\mathcal{H}\cap \elltext=\Sing\F\cap\elltext$;
\vspace{0.5mm}
\item [(iv)] $\forall\hspace{1mm}s\in\Sing\mathcal{H}\cap\elltext,
              \hspace{1mm}
              \mu(\mathcal{H},s)=1$;
\vspace{0.5mm}
\item [(v)] $\forall\hspace{1mm}s\in\Sing\mathcal{H}\cap\elltext,
             \hspace{1mm}
             \tau(\mathcal{H},s)=\tau(\mathcal{\F},s)$;
\vspace{0.5mm}
\item [(vi)] $\forall\hspace{1mm}s\in\Sing\mathcal{H}\cap\elltext,
              \hspace{1mm}
              \mathrm{CS}(\mathcal{H},\elltext,s)=\mathrm{CS}(\F,\elltext,s).$
\end{itemize}
}
\end{pro}

\begin{proof}
We take a homogeneous coordinate system $[x:y:z]\in\pp$ such that $\elltext=\{z=0\}.$ Since $\elltext$ is $\F$-invariant, $\F$ is defined in the affine chart $z=1$ by a $1$-form of the following type $$\omega=\sum_{i=0}^{d}(A_i(x,y)\mathrm{d}x+B_i(x,y)\mathrm{d}y),$$ where $A_i,\,B_i$ are homogeneous polynomials of degree $i$. Using the fact that every reduced convex foliation on $\pp$ has only non-degenerate singularities (\cite[Lemma~6.8]{BM17}) and arguing as in the proof of \cite[Proposition~6.4]{BM17}, we see that the $1$-form $\omega_d=A_d(x,y)\mathrm{d}x+B_d(x,y)\mathrm{d}y$ well defines a homogeneous foliation $\mathcal{H}$ of degree $d$ on $\pp$, and that this foliation satisfies the announced properties (i),~(ii),~(iii),~(iv) and (vi). In particular, since $\F$ is convex by hypothesis, property~(i) implies that $\mathcal{H}$ is also convex.

\noindent Let us show that $\mathcal{H}$ also satisfies property~(v). Set $\Lambda:=\Sing\mathcal{H}\cap \elltext=\Sing\F\cap\elltext$; since $\F$ possesses $3d$ invariant lines, we have  $\sum_{\scalebox{0.85}{$s\hspace{-0.8mm}\in\hspace{-0.8mm}\raisebox{-0.2mm}{$\Lambda$}$}}\big(\sigma(\F,s)-1\big)=3d-1.$
By Lemma~\ref{lem:sigma(F,s)}, this is equivalent to $\sum_{\scalebox{0.85}{$s\hspace{-0.8mm}\in\hspace{-0.8mm}\raisebox{-0.2mm}{$\Lambda$}$}}\tau(\F,s)=3d-1.$ By \cite[Proposition~2.2]{BM17} the convexity of  $\mathcal{H}$ implies that $\sum_{\scalebox{0.85}{$s\hspace{-0.8mm}\in\hspace{-0.8mm}\raisebox{-0.2mm}{$\Lambda$}$}}\tau(\mathcal{H},s)=2d-2.$ Moreover, the already proved property~(iv)
ensures that $\#\Lambda=d+1.$ It follows that
\begin{align*}
&
\sum_{\scalebox{0.85}{$s\hspace{-0.8mm}\in\hspace{-0.8mm}\raisebox{-0.2mm}{$\Lambda$}$}}\big(\tau(\F,s)-1\big)
=(3d-1)-\#\Lambda
=2d-2
=
\sum_{\scalebox{0.85}{$s\hspace{-0.8mm}\in\hspace{-0.8mm}\raisebox{-0.2mm}{$\Lambda$}$}}\big(\tau(\mathcal{H},s)-1\big).
\end{align*}

\noindent Thus, in order to see that  $\mathcal{H}$ satisfies property~(v), it is enough to prove that  $\tau(\mathcal{\F},s)\leq\tau(\mathcal{H},s)$ for each $s\in\Lambda.$
Let us fix $s\in\Lambda.$ Up to conjugating $\omega$ by a linear isomorphism of  $\mathbb{C}^2=(z=1)$, we can assume that  $s=[0:1:0].$ The foliations  $\F$ and $\mathcal{H}$ are respectively defined in the affine chart $y=1$ by the $1$-forms
\begin{align*}
&\theta=\sum_{i=0}^{d}z^{d-i}[A_i(x,1)(z\mathrm{d}x-x\mathrm{d}z)-B_i(x,1)\mathrm{d}z]
&& \text{and} &&
\theta_d=A_d(x,1)(z\mathrm{d}x-x\mathrm{d}z)-B_d(x,1)\mathrm{d}z.
\end{align*}
As a consequence
\begin{align*}
\tau(\F,s)=
\min\Big\{k\geq1\hspace{1mm}\colon J^{k}_{(0,0)}\Big(\sum_{i=0}^{d}z^{d-i}B_i(x,1)\Big)\neq0\Big\}\leq
\min\Big\{k\geq1\hspace{1mm}\colon J^{k}_{0}\big(B_d(x,1)\big)\neq0\Big\}=\tau(\mathcal{H},s).
\end{align*}
\end{proof}

\begin{rem}\label{rem:Darboux-BB}
If $\F$ is a foliation of degree  $d$ on $\pp$ then  (\emph{see} \cite{Bru00})
\begin{align}\label{equa:Darboux-BB}
&\sum_{s\in\mathrm{Sing}\F}\mu(\F,s)=d^2+d+1
&&\text{and}&&
\sum_{s\in\mathrm{Sing}\F}\mathrm{BB}(\F,s)=(d+2)^2.
\end{align}
\end{rem}

\begin{lem}\label{lem:non-radiale}
{\sl Every foliation of degree $d\geq 1$ on  $\pp$ possesses at least a non radial singularity.}
\end{lem}

\noindent This lemma follows from the formulas (\ref{equa:Darboux-BB}) and the obvious following remark: if a foliation $\F$ on $\pp$ admits a radial singularity $s$, then $\mu(\F,s)=1$ and $\mathrm{BB}(\F,s)=4.$

\begin{proof}[\sl Proof of Theorem~\ref{thmalph:convexe-reduit-4}]
Let $\F$ be a reduced convex foliation of degree $4$ on $\pp$. Let us denote by $\Sigma$ the set of non radial singularities of $\F$. By Lemma~\ref{lem:non-radiale}, $\Sigma$ is nonempty. Since by hypothesis $\F$ is reduced convex, all its singularities have
\textsc{Milnor} number $1$ (\cite[Lemma~6.8]{BM17}). The set $\Sigma$ consists then of the singularities $s\in\Sing\F$ such that $\tau(\F,s)=1.$ Let $m$ be a point of $\Sigma$; by Lemma~\ref{lem:sigma(F,s)}, through the point $m$ pass exactly two $\F$-invariant lines  $\elltext_{m}^{(1)}$ and  $\elltext_{m}^{(2)}$.

\noindent On the other hand, for any line $\elltext $ invariant by $\F$, Proposition~\ref{pro:F-degenere-H} ensures the existence of a convex homogeneous foliation $\mathcal{H}_{\ellindice}$ of degree $4$ on $\pp$ belonging to $\overline{\mathcal{O}(\F)}$ and such that the line $\elltext$ is $\mathcal{H}_{\ellindice}$-invariant. Therefore $\mathcal{H}_{\ellindice}$, and in particular each  $\mathcal{H}_{\ellindice_{m}^{(i)}}$, is linearly conjugated to one of the five homogeneous foliations given by Theorem~\ref{thmalph:class-homogene-convexe-4}. Proposition~\ref{pro:F-degenere-H} also ensures that
\vspace{1mm}
\begin{itemize}
\item [($\mathfrak{a}$)] $\Sing\F\cap\elltext=\Sing\mathcal{H}_{\ellindice}\cap\elltext$;
\vspace{0.5mm}
\item [($\mathfrak{b}$)] $\forall\hspace{1mm}s\in\Sing\mathcal{H}_{\ellindice}\cap\elltext,\hspace{1mm}
                          \mu(\mathcal{H}_{\ellindice},s)=1$;
\vspace{0.5mm}
\item [($\mathfrak{c}$)] $\forall\hspace{1mm}s\in\Sing\mathcal{H}_{\ellindice}\cap\elltext,\hspace{1mm}
                         \tau(\mathcal{H}_{\ellindice},s)=\tau(\F,s)$;
\vspace{0.5mm}
\item [($\mathfrak{d}$)] $\forall\hspace{1mm}s\in\Sing\mathcal{H}_{\ellindice}\cap\elltext,\hspace{1mm}
                         \mathrm{CS}(\mathcal{H}_{\ellindice},\elltext,s)=\mathrm{CS}(\F,\elltext,s)$.
\end{itemize}
\vspace{1mm}

\noindent Since $\mathrm{CS}(\F,\elltext_{m}^{(1)},m)\mathrm{CS}(\F,\elltext_{m}^{(2)},m)=1,$ relation ($\mathfrak{d}$) implies that $\mathrm{CS}(\mathcal{H}_{\ellindice_{m}^{(1)}},\elltext_{m}^{(1)},m)\mathrm{CS}(\mathcal{H}_{\ellindice_{m}^{(2)}},\elltext_{m}^{(2)},m)=1.$
This equality and Table~\ref{tab:CS(lambda)} lead to
\begin{align*}
&\{\mathrm{CS}(\mathcal{H}_{\ellindice_{m}^{(1)}},\elltext_{m}^{(1)},m),\,\mathrm{CS}(\mathcal{H}_{\ellindice_{m}^{(2)}},\elltext_{m}^{(2)},m)\}=\{-3,-\tfrac{1}{3}\}
&&\text{or}&&
\mathrm{CS}(\mathcal{H}_{\ellindice_{m}^{(1)}},\elltext_{m}^{(1)},m)=\mathrm{CS}(\mathcal{H}_{\ellindice_{m}^{(2)}},\elltext_{m}^{(2)},m)=-1.
\end{align*}

\noindent At first let us suppose that it is possible to choose $m\in\Sigma$ so that
$$
\{\mathrm{CS}(\mathcal{H}_{\ellindice_{m}^{(1)}},\elltext_{m}^{(1)},m),\,\mathrm{CS}(\mathcal{H}_{\ellindice_{m}^{(2)}},\elltext_{m}^{(2)},m)\}
=\{-3,-\tfrac{1}{3}\}.$$
By renumbering the $\elltext_{m}^{(i)}$ we can assume that $\mathrm{CS}(\mathcal{H}_{\ellindice_{m}^{(1)}},\elltext_{m}^{(1)},m)=-\frac{1}{3}$ and $\mathrm{CS}(\mathcal{H}_{\ellindice_{m}^{(2)}},\elltext_{m}^{(2)},m)=-3.$ Consulting Table~\ref{tab:CS(lambda)}, we see that
\begin{align*}
&\hspace{5mm}\mathcal{T}_{\mathcal{H}_{\ellindice_{m}^{(1)}}}=2\cdot\mathrm{R}_3,&&
\mathcal{T}_{\mathcal{H}_{\ellindice_{m}^{(2)}}}\in\Big\{ 3\cdot\mathrm{R}_1+1\cdot\mathrm{R}_3,\hspace{0.2mm}2\cdot\mathrm{R}_1+2\cdot\mathrm{R}_2\Big\}.
\end{align*}
Therefore, it follows from relations ($\mathfrak{a}$) and ($\mathfrak{c}$) that $\F$ possesses two radial singularities $m_1,m_2$ of order $3$ on the line  $\elltext_{m}^{(1)}$ and a radial singularity  $m_3$ of order $2$ or $3$ on the line $\elltext_{m}^{(2)}.$

\noindent We will see that the radiality order of the singularity $m_3$ of $\F$ is necessarily $3,$  {\it i.e.}  $\tau(\F,m_3)=4.$
By \cite[Proposition~2, page~23]{Bru00}, the fact that  $\tau(\F,m_1)+\tau(\F,m_3)\geq4+3>\deg\F$ implies the invariance by $\F$ of the line $\elltext=(m_1m_3)$; if $\tau(\F,m_3)$ were equal to $3$, then relations ($\mathfrak{a}$), ($\mathfrak{b}$) and ($\mathfrak{c}$), combined with the convexity of the foliation $\mathcal{H}_{\ellindice}$, would imply that $\mathcal{T}_{\mathcal{H}_{\ellindice}}=1\cdot\mathrm{R}_1+1\cdot\mathrm{R}_2+1\cdot\mathrm{R}_3$ so that (\emph{see} Table~\ref{tab:CS(lambda)}) $\mathcal{H}_{\ellindice}$ would possess a singularity $m'$ on the line  $\elltext$ satisfying $\mathrm{CS}(\mathcal{H}_{\ellindice},\elltext,m')\in\{-\frac{13+2\sqrt{13}}{13},-\frac{13-2\sqrt{13}}{13}\}$ which is not possible.

\noindent By construction, the three points  $m_1$, $m_2$ and $m_3$ are not aligned. We have thus shown that $\F$ admits three non-aligned radial singularities of order $3$. By~\cite[Proposition~6.3]{BM17} the foliation $\F$ is linearly conjugated to the \textsc{Fermat} foliation $\F_{0}^{4}.$
\smallskip

\noindent Let us now consider the eventuality $\mathrm{CS}(\mathcal{H}_{\ellindice_{m}^{(1)}},\elltext_{m}^{(1)},m)=\mathrm{CS}(\mathcal{H}_{\ellindice_{m}^{(2)}},\elltext_{m}^{(2)},m)=-1$
for any choice of $m\in\Sigma.$ In~this case, Table~\ref{tab:CS(lambda)} leads to  $\mathcal{T}_{\mathcal{H}_{\ellindice_{m}^{(i)}}}=3\cdot\mathrm{R}_2$ for $i=1,2.$ Then, as before, by using relations ($\mathfrak{a}$), ($\mathfrak{b}$) and ($\mathfrak{c}$), we obtain that $\F$ possesses exactly three radial singularities of order $2$ on each line  $\elltext_{m}^{(i)}$. Moreover, every line joining a radial singularity of order $2$ of $\F$ on $\elltext_{m}^{(1)}$ and a radial singularity of order $2$ of $\F$ on $\elltext_{m}^{(2)}$ must contain necessarily a third radial singularity of order~$2$~of~$\F.$
We can then choose a homogeneous coordinate system $[x:y:z]\in\pp$ so that the points $m_1=[0:0:1],\,m_2=[1:0:0]$\, and \,$m_3=[0:1:0]$ are radial singularities of order  $2$ of $\F$. Moreover, in this coordinate system the lines $x=0$, $y=0$, $z=0$ must be invariant by $\F$ and there exist $x_0,y_0,z_0\in \mathbb{C}^{*}$ such that  the points $m_4=[x_0:0:1],\,m_5=[1:y_0:0],\,m_6=[0:1:z_0]$ are radial singularities of order $2$ of $\F.$ The equalities $\nu(\F,m_1)=1$, $\tau(\F,m_1)=3$ and the invariance of the line $z=0$ by $\F$ ensure that every $1$-form $\omega$ defining $\F$ in the affine chart $z=1$ is of type
\begin{align*}
&\omega=(x\mathrm{d}y-y\mathrm{d}x)(\gamma+c_0\hspace{0.2mm}x+c_1y)+(\alpha_0\hspace{0.2mm}x^3+\alpha_1x^2y+\alpha_2xy^2+\alpha_3y^3)\mathrm{d}x
+(\beta_0\hspace{0.2mm}x^3+\beta_1x^2y+\beta_2xy^2+\beta_3y^3)\mathrm{d}y\\
&\hspace{0.7cm}+(a_0\hspace{0.2mm}x^4+a_1x^3y+a_2x^2y^2+a_3xy^3+a_4y^4)\mathrm{d}x+(b_0\hspace{0.2mm}x^4+b_1x^3y+b_2x^2y^2+b_3xy^3+b_4y^4)\mathrm{d}y,
\end{align*}
where $a_i,b_i,c_{\hspace{-0.2mm}j},\alpha_k,\beta_k\in\mathbb{C}$ and $\gamma\in\mathbb{C}^*.$
\vspace{1mm}

\noindent In the affine chart $x=1$, resp. $y=1,$ the foliation $\F$ is given by
\begin{align*}
&\theta=z^3(\gamma\hspace{0.1mm}z+c_0+c_1y)\mathrm{d}y-(\alpha_0\hspace{0.2mm}z+\alpha_1yz+\alpha_2y^2z+\alpha_3y^3z+a_0+a_1y+a_2y^2+a_3y^3+a_4y^4)\mathrm{d}z\\
&\hspace{0.62cm}-(\beta_0\hspace{0.2mm}z+\beta_1yz+\beta_2y^2z+\beta_3y^3z+b_0+b_1y+b_2y^2+b_3y^3+b_4y^4)(y\mathrm{d}z-z\mathrm{d}y),\\
&\hspace{-0.9cm}\text{resp.}\hspace{1.5mm}
\eta=-z^3(\gamma\hspace{0.1mm}z+c_0\hspace{0.2mm}x+c_1)\mathrm{d}x
-(\beta_0\hspace{0.2mm}x^3z+\beta_1x^2z+\beta_2xz+\beta_3z+b_0\hspace{0.2mm}x^4+b_1x^3+b_2x^2+b_3x+b_4)\mathrm{d}z\\
&\hspace{0.62cm}+(\alpha_0\hspace{0.2mm}x^3z+\alpha_1x^2z+\alpha_2xz+\alpha_3z+a_0\hspace{0.2mm}x^4+a_1x^3+a_2x^2+a_3x+a_4)(z\mathrm{d}x-x\mathrm{d}z).
\end{align*}
A straightforward computation shows that
\begin{small}
\begin{align*}
&\Big(J^{2}_{(y,z)=(0,0)}\theta\Big)\wedge\Big(y\mathrm{d}z-z\mathrm{d}y\Big)=-z\hspace{0.1mm}P(y,z)\mathrm{d}y\wedge\mathrm{d}z,&&
\Big(J^{2}_{(x,z)=(0,0)}\eta\Big)\wedge\Big(z\mathrm{d}x-x\mathrm{d}z\Big)=z\hspace{0.1mm}Q(x,z)\mathrm{d}x\wedge\mathrm{d}z,\\
&\Big(J^{2}_{(x,y)=(x_0,0)}\omega\Big)\wedge\Big((x-x_0)\mathrm{d}y-y\mathrm{d}x\Big)=x_0\hspace{0.1mm}R(x,y)\mathrm{d}x\wedge\mathrm{d}y,&&
\Big(J^{2}_{(y,z)=(y_0,0)}\theta\Big)\wedge\Big((y-y_0)\mathrm{d}z-z\mathrm{d}y\Big)=-z\hspace{0.1mm}S(y,z)\mathrm{d}y\wedge\mathrm{d}z,\\
&\Big(J^{2}_{(x,z)=(0,z_0)}\eta\Big)\wedge\Big((z-z_0)\mathrm{d}x-x\mathrm{d}z\Big)=T(x,z)\mathrm{d}x\wedge\mathrm{d}z
\end{align*}
\end{small}
\hspace{-1mm}with
\begin{small}
\begin{align*}
&P(y,z)=a_0+a_1y+\alpha_0z+a_2y^2+\alpha_1yz,\\
&Q(x,z)=b_4+b_3x+\beta_3z+b_2x^2+\beta_2xz,\\
&R(x,y)=-x_0^3(\alpha_0+3a_0\hspace{0.2mm}x_0)+x_0^2(4\alpha_0+11a_0\hspace{0.2mm}x_0)x
+(\gamma+\alpha_1x_0^2+\beta_0\hspace{0.2mm}x_0^2+2a_1x_0^3+3b_0\hspace{0.2mm}x_0^3)y-2x_0(3\alpha_0+7a_0\hspace{0.2mm}x_0)x^2\\
&\hspace{1.32cm}+(c_0-3\alpha_1x_0-3\beta_0\hspace{0.2mm}x_0-5a_1x_0^2-8b_0\hspace{0.2mm}x_0^2)xy+(c_1-\alpha_2x_0-\beta_1x_0-a_2x_0^2-2b_1x_0^2)y^2
+(3\alpha_0+6a_0\hspace{0.2mm}x_0)x^3\\
&\hspace{1.32cm}+(2\alpha_1+3\beta_0+3a_1x_0+6b_0\hspace{0.2mm}x_0)x^2y+(\alpha_2+2\beta_1+a_2x_0+3b_1x_0)xy^2+(\beta_2+b_2x_0)y^3,\\
&S(y,z)=a_0+b_0y_0+a_3y_0^3+3a_4y_0^4+b_3y_0^4+3b_4y_0^5+(a_1+b_1y_0-3a_3y_0^2-8a_4y_0^3-3b_3y_0^3-8b_4y_0^4)y\\
&\hspace{1.24cm}+(\alpha_0+\beta_0y_0-\alpha_2y_0^2-2\alpha_3y_0^3-\beta_2y_0^3-2\beta_3y_0^4)z+(a_2+3a_3y_0+b_2y_0+6a_4y_0^2+3b_3y_0^2+6b_4y_0^3)y^2\\
&\hspace{1.24cm}+(\alpha_1+2\alpha_2y_0+\beta_1y_0+3\alpha_3y_0^2+2\beta_2y_0^2+3\beta_3y_0^3)yz,\\
&T(x,z)=-b_4z_0-z_0(a_4+b_3-c_1z_0^2-3\gamma\hspace{0.1mm}z_0^3)x+(b_4-\beta_3z_0)z-z_0(a_3+b_2+\beta_1z_0+2c_0z_0^2)x^2
+(\beta_2+3c_1z_0+6\gamma\hspace{0.1mm}z_0^2)xz^2\\
&\hspace{1.32cm}+(b_3-\alpha_3z_0-\beta_2z_0-3c_1z_0^2-8\gamma\hspace{0.1mm}z_0^3)xz+\beta_3z^2-z_0(a_2+\alpha_1z_0)x^3+(b_2-\alpha_2z_0+\beta_1z_0+3c_0z_0^2)x^2z,
\end{align*}
\end{small}
\hspace{-1.6mm}so that the equality  $\tau(\F,m_2)=3$ (resp. $\tau(\F,m_3)=3$, resp. $\tau(\F,m_4)=3$, resp. $\tau(\F,m_5)=3$, resp. $\tau(\F,m_6)=3$) implies that the polynomial $P$ (resp. $Q$, resp. $R$, resp. $S$, resp. $T$) is identically zero. From  $P=Q=0$ we obtain  $a_0=a_1=a_2=\alpha_0=\alpha_1=b_4=b_3=b_2=\beta_3=\beta_2=0$. Next, from the equalities $R=S=T=0$ we deduce that
\begin{align*}
&
c_0=2\gamma\hspace{0.1mm}y_0z_0(x_0y_0z_0+1),&&
c_1=-2\gamma\hspace{0.1mm}z_0,&&
\alpha_2=2\gamma\hspace{0.1mm}y_0z_0^2(x_0y_0z_0+2),&&
\alpha_3=-2\gamma\hspace{0.1mm}z_0^2,&&
\beta_0=2\gamma\hspace{0.1mm}x_0y_0^3z_0^3,\\
&
\beta_1=-2\gamma\hspace{0.1mm}y_0z_0^2(2x_0y_0z_0+1),&&
a_3=-2\gamma\hspace{0.1mm}y_0z_0^3,&&
a_4=\gamma\hspace{0.1mm}z_0^3,&&
b_0=-\gamma\hspace{0.1mm}y_0^3z_0^3,&&
b_1=2\gamma\hspace{0.1mm}y_0^2z_0^3,\\
&
(x_0y_0z_0)^{2}+x_0y_0z_0+1=0.
\end{align*}
Let us set $\rho=x_0y_0z_0$; then $\rho^2+\rho+1=0$ and $\omega$ is of type
\begin{align*}
&\omega=\gamma\Big(x\mathrm{d}y-y\mathrm{d}x\Big)\Big(1+2y_0z_0(\rho+1)x-2z_0y\Big)
+2\gamma\hspace{0.1mm}z_0^2y^2\Big(y_0(\rho+2)x-y\Big)\mathrm{d}x
+\gamma\hspace{0.1mm}z_0^3y^3\Big(y-2y_0\hspace{0.1mm}x\Big)\mathrm{d}x\\
&\hspace{0.68cm}+2\gamma\hspace{0.1mm}y_0z_0^2\hspace{0.1mm}x^2\Big(y_0\rho x-(2\rho+1)y\Big)\mathrm{d}y
+\gamma\hspace{0.1mm}y_0^2z_0^3x^3\Big(2y-y_0\hspace{0.1mm}x\Big)\mathrm{d}y.
\end{align*}
This $1$-form is linearly conjugated to
\[
\omegahesse=(2x^{3}-y^{3}-1)y\mathrm{d}x+(2y^{3}-x^{3}-1)x\mathrm{d}y\hspace{1mm}.
\]
Indeed, the fact that $\rho$ satisfies $\rho^2+\rho+1=0$ implies that
\[
\hspace{1.3cm}\omegahesse=\frac{\raisebox{-0.5mm}{$9$}y_0z_0^2}{\gamma(\rho-1)}\varphi^*\omega,
\quad \text{where}\hspace{1.5mm}
\varphi=\left(\frac{2\rho+1-(\rho+2)x-(\rho+2)y}{3y_0z_0},\frac{(\rho-1)x-(2\rho+1)y+\rho+2}{3z_0}\right).
\]
\end{proof}


\end{document}